\documentclass[a4paper]{amsart}
\usepackage{a4wide}
\usepackage{amssymb, amsmath,amsthm,amscd}
\usepackage{hyperref}
\usepackage{ifthen}
\usepackage{enumitem}

\usepackage{tikz-cd}

\usepackage{yfonts}
\usepackage[T1]{fontenc}

\theoremstyle{plain}
\newtheorem{thm}{Theorem}[section]
\newtheorem*{thm*}{Theorem}
\newtheorem{prop}[thm]{Proposition}
\newtheorem{lem}[thm]{Lemma}
\newtheorem*{lem*}{Lemma}
\newtheorem{cor}[thm]{Corollary}

\theoremstyle{remark}
\newtheorem*{remark}{Remark}

\theoremstyle{definition}
\newtheorem{dfn}[thm]{Definition}

\DeclareMathOperator{\Br}{Br}
\DeclareMathOperator{\Pic}{Pic}
\DeclareMathOperator{\Hom}{Hom}
\DeclareMathOperator{\Gal}{Gal}
\DeclareMathOperator{\inv}{inv}
\DeclareMathOperator{\dP5}{dP_5}
\DeclareMathOperator{\Spec}{Spec}
\DeclareMathOperator{\Norm}{Nm}

\newcommand{\opp}{^{\textup{opp}}}
\newcommand{\sep}{^{\textup{sep}}}
\newcommand{\kbar}{\bar{k}}
\newcommand{\Xbar}{\bar{X}}
\newcommand{\Ubar}{\bar{U}}
\renewcommand{\H}{\mathrm{H}}
\newcommand{\Gm}{\mathbb{G}_m}
\newcommand{\tors}{_{\textup{tors}}}
\renewcommand{\div}{\mathrm{div}}
\newcommand{\BrAz}{\Br_{\textup{Az}}}
\renewcommand{\Im}{\mathrm{Im}}

\AtBeginDocument{}

\begin{document}

\title[Order $5$ obstructions to the integral Hasse principle]{Order $5$ Brauer--Manin obstructions to the integral Hasse principle on log K3 surfaces}

\author{Julian Lyczak}
\address{IST Austria\\ Am Campus 1\\ 3400 Klosterneuburg\\ Austria}
\email{jlyczak@ist.ac.at}

\begin{abstract}
We construct families of log K3 surfaces and study the arithmetic of their members. We use this to produce explicit surfaces with an order $5$ Brauer--Manin obstruction to the integral Hasse principle.
\end{abstract}

\maketitle

\section*{Introduction}

The goal of this paper is to add to the study of integral points on ample \textit{log K3 surfaces} as started by Harpaz \cite{Harpaz}. This is done by first giving a geometrically flavoured construction for such equations. One upshot of this construction is that one even gets a family of such surfaces for which the arithmetic properties of the members can be studied simultaneously. We will present two families of log K3 surfaces for which a positive proportion of the fibres fails the \textit{Hasse principle}, i.e.~it is everywhere locally soluble but it does not admit integral points.

For the geometrically similar K3 surfaces it has been conjectured by Skorobogatov that the \textit{Brauer--Manin obstruction} is the only one to the Hasse principle \cite{SkorobogatovDiagQuarSurf}. This is however not the case for log K3 surface \cite{Harpaz}. The examples in this paper exhibit new arithmetic behaviour and it is hoped that the accompanying ideas for studying log K3 surfaces will contribute to a workable conjecture for integral points on log K3 surfaces.

\subsection*{Main results} 

We will use the \textit{integral Brauer--Manin obstruction} as introduced by Colliot-Th\'el\`ene and Xu \cite{CTXu} to prove the failure of the integral Hasse principle, which is based on the Brauer--Manin obstruction by Manin \cite{ManinICM}. Let $\mathcal U/\mathbb Z$ be a model of a variety $U/\mathbb Q$ for which we want to prove that $\mathcal U(\mathbb Z)=\emptyset$. The technique uses an element $\mathcal A \in \Br U := \H^2(U,\Gm)$ to define an intermediate set
\[
\mathcal U(\mathbb Z) \subseteq \mathcal U(\mathbb A_{\mathbb Q, \infty})^{\mathcal A} \subseteq \mathcal U(\mathbb A_{\mathbb Q, \infty})
\]
in the inclusion of integral points of $\mathcal U$ in its set of integral adelic points, i.e.~$\mathbb A_{\mathbb Q,\infty}=\mathbb R \times \prod_p \mathbb Z_p$. Hence if $\mathcal U(\mathbb A_{\mathbb Q, \infty})$ is non-empty but $\mathcal U(\mathbb A_{\mathbb Q, \infty})^{\mathcal A}$ is empty, then $\mathcal A$ \textit{obstructs the integral Hasse principle} on $\mathcal U$. The \textit{order of the obstruction} is the order of $\mathcal A$ in $\Br U/\Br \mathbb Q$. We will be mainly interested in obstructions coming from elements in the \textit{algebraic Brauer group} $\Br_1 U := \ker\left( \Br U \to \Br \Ubar \right) \subseteq \Br U$.\\

In this paper we have restricted to a specific type of log K3 surface to showcase our ideas. In passing we pick up the first Brauer--Manin obstructions of higher order. The existence of a high order element in the Brauer group of our log K3 surfaces depends on the splitting field of a related del Pezzo surface. Recall that the splitting field of a del Pezzo surface is the minimal field over which all its $-1$-curves are defined.

\begin{thm}[Theorem~\ref{thm:algbrgrplogdp5}]
Let $U=X\backslash C$ be a log K3 surface over $\mathbb Q$ with $X$ a del Pezzo surface of degree $5$ and $C$ a geometrically irreducible anticanonical divisor.
We have
\[
\Br_1 U /\Br \mathbb Q \cong \begin{cases}
	\mathbb Z/5\mathbb Z & \mbox{ if the splitting field of $X$ is a cyclic extension $K/\mathbb Q$ of degree $5$;}\\ 
	0 & \mbox{ otherwise.}
\end{cases}
\]
Also, each cyclic extension $K/\mathbb Q$ of degree $5$ is the splitting field of a del Pezzo surface over $\mathbb Q$. Such a surface is unique up to isomorphism.
\end{thm}

We will consider del Pezzo surfaces with a non-trivial algebraic Brauer group and our first explicit example comes from the quintic extension $\mathbb Q(\zeta_{11}+\zeta^{-1}_{11})/\mathbb Q$. Consider the projective scheme~$\mathcal X \subseteq \mathbb P^5_{\mathbb Z}$ given by the five quadratic forms
\begin{multline*}
u_0u_3+22u_0u_4+121u_0u_5-u_1^2-121u_1u_3+2662u_1u_4-36355u_2u_4-\\
9306u_2u_5+10494u_3u_4-242u_3u_5-215501u_4^2+68123u_4u_5-13794u_5^2,
\end{multline*}\\[-1.2cm]
\begin{multline*}
u_0u_4+11u_0u_5-u_1u_2-11u_1u_3+242u_1u_4-3223u_2u_4-847u_2u_5+\\
902u_3u_4-11u_3u_5-19272u_4^2+6413u_4u_5-1331u_5^2,
\end{multline*}\\[-1.2cm]
\begin{multline*}
u_0u_5-u_1u_3+22u_1u_4-u_2^2-286u_2u_4-77u_2u_5+77u_3u_4-
1694u_4^2+572u_4u_5-121u_5^2,
\end{multline*}\\[-1.2cm]
\begin{multline*}
u_1u_4-u_2u_3-11u_2u_4-77u_4^2+55u_4u_5-11u_5^2,\\
\end{multline*}\\[-1.7cm]
\begin{multline*}
u_1u_5-u_2u_4-11u_2u_5-u_3^2+11u_3u_4-44u_4^2.\\
\end{multline*}\\[-0.8cm]
\noindent This scheme is constructed and studied in Section~3. In Section~4 we prove the arithmetic properties stated in the following theorem.

\begin{thm}
For each geometrically irreducible hyperplane section $\mathcal C_h := \{h=0\} \cap \mathcal X$ we define~$\mathcal U_h =\mathcal X\backslash \mathcal C_h$.
\begin{enumerate}
\item The scheme $\mathcal X/\mathbb Z$ is a flat proper model of the del Pezzo surface over $\mathbb Q$ which splits over the quintic extension $\mathbb Q(\zeta_{11}+\zeta^{-1}_{11})$.
\item The existence of an algebraic Brauer--Manin obstruction to the integral Hasse principle on $\mathcal U_h$ only depends on the reduction of $h$ modulo $11$.
\item There exists an $h$, and hence even a residue class $h \bmod 11$, for which $\mathcal U_h$ has an order $5$ obstruction to the integral Hasse principle.
\end{enumerate}
\end{thm}

The same construction can be used to produce many more examples. The arithmetic behaviour is mainly determined by the primes which are ramified in the splitting field $K$. Any tamely ramified prime can be studied in a similar matter. For completeness we also add an example in Section~5 involving the wildly ramified prime $5$.

\begin{thm}
There exists a scheme~$\mathcal X \subseteq \mathbb P^5_{\mathbb Z}$ with the following properties.
\begin{enumerate}
\item The scheme $\mathcal X$ is a flat model for the del Pezzo surface $X=\mathcal X_{\mathbb Q}$ over $\mathbb Q$ which splits over the unique quintic number field $K \subseteq \mathbb Q(\zeta_5)$.
\item The existence of an algebraic Brauer--Manin obstruction on $\mathcal U_h := \mathcal X\backslash \{h=0\}$ to the integral Hasse principle only depends on the reduction of $h$ modulo $25$.
\item There exists an $h$, and hence even a residue class $h \bmod 25$, for which $\mathcal U_h$ has an order $5$ obstruction to the integral Hasse principle.
\end{enumerate}
\end{thm}
Let us put these results in context.

\subsection*{Integral points on log K3 surfaces} Integral points on log K3 are believed to behave to a certain degree in a similar way as rational points on K3 surfaces. For those surfaces it has been conjectured by Skorobogatov \cite{SkorobogatovDiagQuarSurf} that the existence of solutions are completely controlled by the Brauer--Manin obstruction. However, results by Ieronymou and Skorobogatov \cite{IeronymouSkorobogatovOddOrder} and Skorobogatov and Zarhin \cite{SkorobogatovZarhinKummerVarietiesBrauerGroups} say that there can not be an odd order obstruction to the Hasse principle for smooth diagonal quartic surfaces and Kummer varieties. An algebraic obstruction of order $3$ on a K3 surface was found in \cite{CornNakahara} and Berg and V\'arilly-Alvarado \cite{BergVarillyAlvarado} even produced a transcendental cubic obstruction.

For log K3 surfaces the situation is however different; it was proven that the Brauer--Manin obstruction is not the only obstruction to the integral Hasse principle \cite{Harpaz} and \cite{JS}. On the other hand, Colliot-Th\'el\`ene and Wittenberg \cite{CTW} showed that the Brauer group never obstructs the Hasse principle for the equation $x^3+y^3+z^3=n$ which is in line with the conjecture that this equation has an integral solution for $n \not \equiv \pm 4\mod 9$.

The Hasse principle and the effectivity of the Brauer--Manin obstruction for the equation $x^3+y^3+z^3-xyz=k$ were studied by Ghosh and Sarnak \cite{GhoshSarnakMarkoffSurfaces}, Colliot-Th\'el\`ene, Wei and Xu \cite{CTWeiXu}, and Loughran and Mitankin \cite{LoughranMitankin}. Another classical affine cubic equation was studied in this manner by Bright and Loughran \cite{BrightLoughran}.

This paper gives the first examples of higher odd order Brauer--Manin obstructions on any type of scheme; all other known examples of the Brauer--Manin obstruction to the (integral) Hasse principle are of either order $2$ or $3$. This is the highest possible prime order for such an obstruction on log K3 surfaces; for a generic anticanonical divisor $C$ on a del Pezzo surface $X$ the order of algebraic Brauer groups of $X \backslash C$ is only divisible by the primes $2$, $3$ and $5$, see Table~1 in \cite{BrightLyczak}. Their results also show that the quintic algebraic obstruction described in this paper are particular to the degree $5$ case; there is an inclusion $\Br X \hookrightarrow \Br_1 U$ whose cokernel is divisible by the degree of the del Pezzo surface $X$. Hence $\Br_1 U/\Br X$ only has $5$-torsion if $X$ is a quintic del Pezzo surface.

\subsection*{A study in families} The novel approach in this paper is to study affine surfaces $U$ in families by fixing the compactification $X$ and letting the complementary divisor $C$ vary. An understanding of the arithmetic and geometry of $X$ will be helpful in studying the open surfaces $U$.

We propose a general methodology for studying this setup, which we illustrated in the special setting of del Pezzo surfaces of degree $5$. An important result for these surfaces is that they always have a point. This is a classic result by Enriques \cite{Enriques1897} which was also proved by Swinnerton-Dyer \cite{SwinnertonDyerdP5s}, Skorobogatov \cite{SkorobogatovdP5s} and many others. This proves that $X$ is rational over $k$, $\Br X/\Br k= 0$ and that $X$ satisfies weak approximation. It also allows one to classify and to construct such surfaces over $k$, which was done in detail in \cite{dP5s}.

We produce models $\mathcal X/\mathbb Z$ for $X/\mathbb Q$ by following this construction over the integers. In this process one has a few more choices along the way which allow one to control the reductions $\mathcal X_{\mathfrak p}$ for all primes $\mathfrak p$. To finally construct a model of a log K3 surface one considers the complement $\mathcal U_h$ of a hyperplane section $\{h=0\}$ in $\mathcal X$.

Using the abundance of points on quintic del Pezzo surfaces we can deduce that for any $h$ the open subscheme $\mathcal U_h$ has points over all $\mathbb Z_\ell$, except possibly for a very few small primes $\ell$. In our cases only local solubility at $\ell=2$ is not immediate and will depend on $h$.

We also use the geometric and arithmetic properties of quintic del Pezzo surfaces to compute the Brauer--Manin obstruction on each $\mathcal U_h$. We show that the invariant maps are identically $0$ for all but an explicit finite list of primes. To effectively deal with a remaining prime $p$ we show that it is enough to only study the closed fibre of $\mathcal X \times \mathbb Z_p$; a surprising result especially for the wildly ramified prime $p=5$. We end up with examples of quintic Brauer--Manin obstructions to both the Hasse principle and strong approximation.

There is no reason why this construction only works for affine opens of quintic del Pezzo surfaces; one could use a similar construction to produce models of rational varieties while controlling the arithmetic of the individual fibres.

\subsection{Outline}

We start by recalling some necessary facts on del Pezzo and log K3 surfaces, Brauer groups and the Brauer--Manin obstruction. Then we compute the algebraic Brauer group of log K3 surfaces $U=X\backslash C$ where $X$ is a del Pezzo surface of degree $5$ and $C$ is an anticanonical divisor. In the third section we give an example of a construction of a model $\mathcal X/\mathbb Z$ of a quintic del Pezzo surface $X/\mathbb Q$ such that any anticanonical complement $U_h=X\backslash \{h=0\}$ has an element of order $5$ in the Brauer group. The next section is devoted to the arithmetic of each $\mathcal U_h$. In particular we compute the Brauer--Manin obstruction coming from the element of order $5$. In the last section we use the same construction to produce a different family of log K3 surface $\mathcal U_h$ for which the arithmetic behaves differently, but there still is an element of order $5$ in the Brauer group.

\subsection{Notation and conventions}

Let $k$ be a field. We will write $\kbar$ for a fixed algebraic closure and $k\sep$ for the separable closure of $k$ in $\kbar$. The absolute Galois group of a field~$k$ is denoted by~$G_k = \Gal(k\sep/k)$. A \textit{variety} over a field~$k$ is a separated scheme of finite type over~$\Spec k$. A \textit{curve} over a field $k$ is a variety over $k$ of pure dimension $1$, it need not be irreducible, reduced or smooth. A \textit{surface} over a field~$k$ is a geometrically integral variety of dimension~$2$ over~$k$. A \textit{curve on a surface} over a field $k$ is a closed subscheme of the surface which is a curve over $k$. For a scheme $X$ over a field $k$ we will write $X_K$ for the base change $X \times_k K$ for any field extension $K$ of $k$.  The notation $\Xbar$ will be synonymous for $X_{\kbar}$.

\subsection*{Acknowledgements}

Most results in these article were obtained during my Ph.D.~studies at Leiden University. I would like to thank my supervisor Martin Bright for his help and our many discussions on the subject.

This paper was completed as part of a project which received funding from the European Union's Horizon 2020 research and innovation
programme under the Marie Sk\l odowska-Curie grant agreement No.~754411.

\section{Preliminaries}

We will consider the existence of integral solutions to polynomial equations defining surfaces. We will first collect the key notions and results on the surfaces we will encounter. Then we will recall the necessary results on Brauer groups and the Brauer--Manin obstruction to integral points.

\subsection{Del Pezzo surface of degree $5$}

We will review some facts about del Pezzo surfaces.  The main references will be \cite{Manin} and \cite{Demazure}. For an overview of the arithmetic of such surfaces one is referred to \cite{arithmeticdps}.

\begin{dfn}
	Let $k$ be a field. A \textit{del Pezzo surface} is a smooth projective surface~$X$ over~$k$ such that the anticanonical line bundle $\omega_X^{-1}$ is ample.  The degree of a del Pezzo surface is the anticanonical self-intersection $K_X^2$.
\end{dfn}

We will only need del Pezzo surfaces of degree $5$. In which case $\omega_X^{-1}$ is even very ample over $k$ and we see that every del Pezzo surface of degree $5$ can be embedded as a degree $5$ surface in $\mathbb P^5_k$. Let us collect some facts on the geometry of these surfaces.

\begin{lem}\label{lem:geomofdp5s}
Let $X$ be a del Pezzo surface of degree $5$ over a separably closed field $k$.
\begin{enumerate}
\item The Picard group $\Pic X$ is free of rank $5$ and it has an orthogonal basis $L_0, L_1, \ldots, L_4$ with respect to intersection pairing which satisfies $L_0^2=1$ and $L_i^2=-1$ for $i\ne 0$.
\item In any such basis the canonical class is given by $K_X=-3L_0+L_1+L_2+L_3+L_4$.
\item There are precisely ten classes $D \in \Pic X$ which satisfy $D^2=D\cdot K_X=-1$, namely $L_i$ for $i\ne 0$ and $L_{ij}:=L_0-L_1-L_2$ for $0<i<j\leq 4$. Each of these classes contain a unique curve, and these curves are smooth, irreducible and have genus $0$.
\item The intersection graph of these ten \textit{$-1$-curves} is the so-called Petersen graph shown in Figure~\ref{fig:petersen}.
	
	\begin{figure}[h]
		\centering
		\begin{tikzpicture}
		\tikzstyle{every node}=[draw, shape=circle, minimum size=1cm];
		\path (0+18:3cm) node (outer0) {$L_1$};
		\path (1*72+18:3cm) node (outer1) {$L_{12}$};
		\path (2*72+18:3cm) node (outer2) {$L_2$};
		\path (3*72+18:3cm) node (outer3) {$L_{23}$};
		\path (4*72+18:3cm) node (outer4) {$L_{14}$};
		\draw (outer0) -- (outer1) -- (outer2) -- (outer3) -- (outer4) -- (outer0);
		\path (0+18:1.4cm) node (inner0) {$L_{13}$};
		\path (1*72+18:1.4cm) node (inner1) {$L_{34}$};
		\path (2*72+18:1.4cm) node (inner2) {$L_{24}$};
		\path (3*72+18:1.4cm) node (inner3) {$L_3$};
		\path (4*72+18:1.4cm) node (inner4) {$L_4$};
		\draw (inner0) -- (inner2) -- (inner4) -- (inner1) -- (inner3) -- (inner0);
		\draw (inner0) -- (outer0);
		\draw (inner1) -- (outer1);
		\draw (inner2) -- (outer2);
		\draw (inner3) -- (outer3);
		\draw (inner4) -- (outer4);
		\end{tikzpicture}
		\caption{The intersection graph of $-1$-curves on a del Pezzo surface of degree $5$.}
		\label{fig:petersen}
	\end{figure}
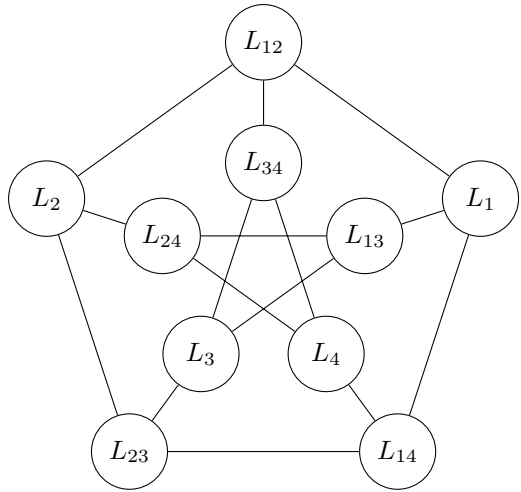
\item Let $A(X)$ be the group of automorphisms of $\Pic X$ which preserve the intersection pairing and the canonical class. Then $A(X)$ is isomorphic to $S_5$ and this isomorphism is unique up to conjugation.
\end{enumerate}
\end{lem}

\begin{proof}
The first result follows from the fact that a del Pezzo surfaces of degree $5$ is geometrically the blowup of the projective plane in $4$ points, no three of which lie on a line, see for example \cite[III, Proposition~3]{Demazure}. From here one can deduce the remaining statements.

We do draw attention to a particularly nice proof of the last statement.  Note that $A(X)$ permutes the $-1$-classes and that these classes generate the Picard group. So $A(X)$ is a subgroup of automorphism group of the Petersen graph. To compute this automorphism group we identify the vertices with $\binom I2$ for $I=\{1,2,3,4,5\}$ such that $\left\{\{i,j\},\{k,l\}\right\}$ is an edge precisely if $i$, $j$, $k$ and $l$ are distinct. This makes it straight forward to show that the automorphism group of the Petersen graph is isomorphic to $S_5$. Next one checks that each automorphism of the graph is actually an automorphism of the whole Picard group.

These isomorphisms are all unique up to conjugacy since $S_5$ is an inner group.
\end{proof}

Let us now switch to del Pezzo surfaces of degree $5$ over general fields. The following proposition shows that the geometric Picard group as a Galois module is a principal invariant.

\begin{prop}\label{prop:bijectionisomdp5s}
Let $k$ be a field. There is a bijection between isomorphism classes of del Pezzo surfaces of degree $5$ over $k$ and conjugacy classes of actions of $G_k$ on $\Pic X\sep$.
\end{prop}

\begin{proof}
This is Lemma~14 in \cite{dP5s}.
\end{proof}

Lemma~\ref{lem:geomofdp5s} and Proposition~\ref{prop:bijectionisomdp5s} together show that there is a bijection between quintic del Pezzo surface over $k$ and group homomorphisms $G_k \to S_5$ up to conjugacy. We will describe how to construct a del Pezzo surface from such a group homomorphism as was done in \cite{dP5s}.

\begin{prop}\label{prop:constructiondp5s}
Let $k$ be a field with absolute Galois group $G_k$ and let $\Lambda$ be the effective generator of $\Pic \mathbb P^2_k$. Consider a group homomorphism $\psi \colon G_k \to S_5$. Fix five points $P_i \in \mathbb P^2_k(k\sep)$ such that no three lie on a line and that $G_k$ acts on these points as $S_5$ acts on its indexes.
\begin{enumerate}
\item The linear system $\mathcal L=\left| \mathcal O_{\mathbb P^2_k}(5\Lambda-2P_1-2P_2-2P_3-2P_4-2P_5) \right|$ has dimension $5$.
\item The image of the associated map is a del Pezzo $X$ surface of degree $5$.
\item The isomorphism class of $X$ only depends on the conjugacy class of $\psi$ and is independent of the choice of $P_i$.
\item The composition $G_k \to A(X\sep) \xrightarrow{\cong} S_5$ recovers $\psi$ up to conjugacy.
\end{enumerate}
\end{prop}

\begin{proof}
The first two statements are Theorem~5 in \cite{dP5s}. It also shows that the $-1$-curves on $X$ correspond to the lines on $\mathbb P^2_{k\sep}$ passing through two of the points $P_i$. This shows that the action of Galois on the $-1$-curves on $X$ and hence on $\Pic X\sep$ equals $\psi$ up to conjugacy.
\end{proof}

\subsection{Log K3 surfaces of $\dP5$ type}

For our interest in integral points we move to surfaces which are not necessarily projective. The following class will be important.

\begin{dfn}
Let~$U$ be a smooth surface over a field~$k$.  A \textit{log K3 structure} on~$U$ is a triple~$(X,C,i)$ consisting of a proper smooth surface~$X$ over $k$, an effective anticanonical divisor $C$ on $X$ with simple normal crossings and an open embedding~$i\colon U \to X$, such that~$i$ induces an isomorphism between~$U$ and~$X\backslash C$. A \textit{log K3 surface} is a simply connected, smooth surface~$U$ over~$k$ together with a choice of log K3 structure~$(X,C,i)$ on~$U$.

Let $X$ be a del Pezzo surface of degree $5$ and let $C$ be an effective anticanonical divisor on $X$. The affine surface $U=X\backslash C$ is called a \textit{log K3 surface of $\dP5$ type}.
\end{dfn}

Whenever we consider such a surface $U$ without explicitly specifying $X$ we will assume the choice of compactification to be understood from context.

\subsection{Brauer groups}

Let~$U$ be a scheme over a field~$k$. We will need the concept of the \textit{Brauer group} $\Br U$ of~$U$.  Two common definitions are the \'etale cohomology group $\Br U := \H^2(U,\Gm)$ and the group $\BrAz U$ of equivalence classes of Azumaya algebras on $U$.  There is a natural morphism $\BrAz U \to \Br U$ which induces an isomorphism between $\BrAz U$ and $(\Br U)\tors$ if $U$ is a quasi-projective scheme over~$k$ by an unpublished result by Gabber.  Another proof by De Jong can be found in \cite{deJong}.  In Theorem~6.6.7 in \cite{Poonen} we find conditions for $\Br U$ to be a torsion group and we conclude that we can identify both types of Brauer groups for regular integral schemes which are quasi-projective over a field.  All the varieties for which we will consider the Brauer group will satisfy these conditions and we will pass freely between the two notions.

Using the functoriality of associating the Brauer group to the scheme we can define the following filtration: $\Br_0 U \subseteq \Br_1 U \subseteq \Br U$, where the \textit{constant Brauer group} $\Br_0 U$ is defined as $\Im(\Br k \to \Br U)$ and the \textit{algebraic Brauer group} $\Br_1 U$ is $\ker(\Br U \to \Br U\sep)$.  We will denote the quotient $\Br_1 U/\Br_0 U$ by $\Br_1 U/\Br k$ although the map $\Br k \to \Br_1 U$ need not be injective.  It follows from the Hochschild--Serre spectral sequence that $\Br_1 U/\Br k$ is isomorphic to~$\H^1(G_k, \Pic U\sep)$ in certain cases.  This is well-known if $U$ is proper, see for example \cite[Corollary~6.7.8]{Poonen}, but the proof actually works under the weaker condition~$\Gm(U\sep)=k^{\textup{sep},\times}$.

If $U$ is an integral noetherian regular scheme over a field of characteristic $0$ the natural map $\Br U \to \Br \kappa(U)$ is an inclusion \cite[Section~II.1]{G-III}. So in this case we can represent elements of the Brauer group by classes of central simple algebras over the field~$\kappa(U)$.  We will construct Azumaya algebras on $U$ as cyclic algebras over the function field.

\begin{dfn}
Let $\kappa$ be a field and $n$ an integer not dividing the characteristic of $\kappa$. An Azumaya algebra in the image of the cup product
\[
\H^1(\kappa,\mu_n) \times \H^1(\kappa,\mathbb Z/n\mathbb Z) \to \H^2(\kappa,\mu_n)\cong \Br k[n]
\]
is called a \textit{cyclic algebra} over $\kappa$.

A cyclic extension $\kappa'/\kappa$ of degree $n$ with a fixed generator $\sigma \in \Gal(\kappa'/\kappa)$ determines an element of $\Hom(G_\kappa,\mathbb Z/n\mathbb Z) \cong \H^1(\kappa,\mathbb Z/n\mathbb Z)$ by sending $\sigma$ to $1$. Any element $a \in \kappa^\times$ gives an element in $\H^1(\kappa,\mu_n) \cong \kappa^\times/(\kappa^\times)^n$. The cyclic algebra $a \cup (\kappa'/\kappa,\sigma)$ is denoted by $(a,\kappa',\sigma)$.
\end{dfn}

For more details, see \cite[Section~1.5.7]{Poonen}. Here one also finds the following important result.

\begin{lem}\label{lem:trivialcycalgs}
A cyclic algebra $(a,\kappa',\sigma)$ is trivial in $\Br \kappa$ precisely when $a \in N_{\kappa'/\kappa}(\kappa'^\times)$.
\end{lem}

To see if a cyclic algebra in~$\Br \kappa(U)$ comes from~$\Br U$ we have the following lemma.

\begin{lem}\label{lem:cyclicalgebraoverscheme}
Consider a smooth and geometrically integral variety $U$ over a field~$k$ satisfying~$\Gm(U\sep)=k^{\textup{sep},\times}$.  Fix a finite cyclic extension $K/k$, a generator $\sigma \in \Gal(K/k)$, and an element~$g \in \kappa(U)^\times$.

The cyclic algebra $\mathcal A = (g, \kappa(U_K)/\kappa(U), \sigma)$ lies in the image of $\Br U \to \Br \kappa(U)$ precisely if $\div g = \Norm_{K/k}(D)$ for some divisor $D$ on $U_K$.
If $k$, and hence $K$, is a number field, and $U$ is everywhere locally soluble then $\mathcal A$ is constant exactly when~$D$ can be taken to be principal.
\end{lem}

\begin{proof}
This lemma is similar to Proposition~4.17 from \cite{Brightthesis}.  The difference is that the projectivity assumption is replaced by the weaker condition~$\Gm(U\sep)=k^{\textup{sep},\times}$.  One can check that under this assumption the proof presented in \cite{Brightthesis} is still valid.
\end{proof}

\subsection{Brauer--Manin obstruction}

In some cases elements of the Brauer group allow us to prove that there are no integral points on a scheme. Let $\mathcal U/\mathbb Z$ be a model of $U=\mathcal U_{\mathbb Q}$. The \textit{Brauer--Manin set} of~$\mathcal A$ is the subset of the \textit{integral adelic points}~$\mathcal U(\mathbb A_{\mathbb Q,\infty}) = U(\mathbb R) \times \prod_\ell \mathcal U(\mathbb Z_\ell)$ defined by
\[
\mathcal U(\mathbb A_{\mathbb Q,\infty})^\mathcal A = \left\{(P_\ell) \in \mathcal U(\mathbb A_{\mathbb Q,\infty}) \ \Bigg|\ \sum_\ell \inv_\ell \mathcal A(P_\ell) = 0 \right\}.
\]

Here the \textit{invariant maps} $\inv_\ell$ are those defined in \cite[Theorem~1.5.34]{Poonen}.  Note that the infinite sum is well-defined by \cite[Proposition~8.2.1]{Poonen}.  The Brauer--Manin set is of particular interest because of the property described in the following theorem from \cite{CTXu}.

\begin{lem}\label{lem:computealgebraicbrauergroup}
Let $\mathcal U$ be a scheme over the integers and let $U$ be the generic fibre over $\mathbb Q$.  For any element $\mathcal A \in \Br U$ we have the following chain of inclusions
\[
\mathcal U(\mathbb Z) \subseteq \mathcal U(\mathbb A_{\mathbb Q,\infty})^\mathcal A \subseteq \mathcal U(\mathbb A_{\mathbb Q,\infty}).
\]
\end{lem}

When $\mathcal A$ is a cyclic algebra we can use Lemma~\ref{lem:trivialcycalgs} to compute the images of the invariant maps and we might gain some information on the set of integral points.

\begin{dfn}
We say that an element~$\mathcal A \in \Br U$ \textit{obstructs the integral Hasse principle} if~$\mathcal U(\mathbb A_{\mathbb Q,\infty})$ is non-empty, but~$\mathcal U(\mathbb A_{\mathbb Q,\infty})^\mathcal A$ is empty.  The \textit{order} of the obstruction is the order of~$\mathcal A$ in~$\Br U/\Br \mathbb Q$.
\end{dfn}

\section{The interesting Galois action}\label{sec:intgalact}

The main goal will be to construct affine schemes $\mathcal U \subseteq \mathbb A^5_{\mathbb Z}$ which have a Brauer--Manin obstruction to the integral Hasse principle. In all our examples we will construct $\mathcal U$ in such a way that $U=\mathcal U_{\mathbb Q}$ is a log K3 surface of $\dP5$\ type. This means that we will be interested in the Brauer group of such surfaces. The following terminology will turn out to be helpful in that regard.

\begin{dfn}
	Let $X$ be a del Pezzo surface of degree $5$ over a field $k$.  Let $K$ be the minimal Galois extension of $k$ over which all $-1$-curves on~$X$ are defined.  We say that $X$ is \textit{interesting} if $[K \colon k] = 5$.  A log K3 surface of $\dP5$\ type $U=X\backslash C$ is called \textit{interesting} if $X$ is an interesting del Pezzo surface and $C$ is geometrically irreducible.
	
	The field $K$ is called the \textit{splitting field} of the interesting surfaces $X$ and $U$.
\end{dfn}

Consider an interesting log K3 surface $U=X\backslash C$.  By definition of a log K3 surface we see that $C$ is smooth. The curve $C$ is also geometrically irreducible since $U$ is interesting.  The results in this paper are also true for the complement of a geometrically irreducible anticanonical curve $C$ on a del Pezzo surface $X$ of degree $5$.  To be able to use the language of log K3 surface we do keep the superfluous condition that $C$ is smooth.

The following lemma shows that an interesting action corresponds to a unique conjugacy class of subgroups of $W_4$.

\begin{lem}\label{lem:anticandivsondp5}
	Consider an interesting del Pezzo surface~$X$ over a field $k$.  The action of~$G_k$ on $\Pic X\sep$ is uniquely determined up to conjugacy.
	
	On an interesting del Pezzo surface there are two Galois orbits of geometric~$-1$-curves, each of size~$5$.  The sum of the~$-1$-curves in one such orbit is an anticanonical divisor.
\end{lem}

\begin{proof}
	Let $K$ be the splitting field of $X$.  Since $X$ is interesting the extension $K/k$ is by definition of degree $5$.	It follows from the minimality of~$K$ that~$\Gal(K/k)$ does not fix any of the ten $-1$-curves, hence there must be two orbits of size~$5$.  After choosing a possibly different basis of $\Pic X\sep$ we see that these two orbits are the two regular pentagons in Figure~\ref{fig:petersen} and that there is a~$\sigma \in \Gal(K/k)$ which acts on the outer pentagon by rotating counter-clockwise.  Since $\sigma$ preserves the intersection pairing it will also rotate the inner pentagon counter-clockwise.  This determines the action of~$\sigma$ on the $-1$-classes:
	\[
	L_1 \mapsto L_{12} \mapsto L_2 \mapsto L_{23} \mapsto L_{14} \mapsto L_1,
	\]
	\[
	L_3 \mapsto L_4 \mapsto L_{13} \mapsto L_{34} \mapsto L_{24} \mapsto L_3.
	\]
	This proves that $L_0=L_{12}+L_1+L_2$ gets mapped to~$2L_0-L_1-L_2-L_3$.  For a different choice of such a basis we get a conjugate action of $G_k$ on $\Pic X\sep$ by Lemma~\ref{lem:geomofdp5s}.
	
	The last statement is immediate.
\end{proof}

If we consider the complement $U$ of a geometrically irreducible anticanonical divisor $C$ on a del Pezzo surface of degree $5$ over a number field $k$ we can compute its algebraic Brauer group modulo constants as $\H^1(G_k, \Pic U\sep)$.  The following proposition shows that the action of $G_k$ on $\Pic X\sep$ is interesting precisely when $\Br_1 U/\Br k$ is non-trivial.

\begin{thm}\label{thm:algbrgrplogdp5}
	Let~$U=X\backslash C$ be a log K3 surface of $\dP5$\ type over a number field~$k$ with $C$ geometrically irreducible. We have
	\[
	\Br_1 U/\Br k \cong \begin{cases}
	\mathbb Z/5\mathbb Z & \mbox{ if~$U$ is interesting;}\\ 
	0 & \mbox{ otherwise.}
	\end{cases}
	\]
\end{thm}

\begin{proof}
	It was mentioned in \cite[Remark at the end of Section~2.1]{BrightLyczak} that the algebraic Brauer group modulo constants of log K3 surfaces of $\dP5$ type with a geometrically irreducible anticanonical divisor $C$ is trivial except for one specific action of the Galois group on the geometric Picard group. So it suffices to verify the statement for interesting del Pezzo surfaces over $k$.  We will fix a basis $(L_0,L_1,L_2,L_3, L_4)$ of $\Pic \Xbar$ as in the proof of Lemma~\ref{lem:anticandivsondp5}.
	
	Since~$C \subseteq X$ is geometrically irreducible we find the following exact sequence of Galois modules
	\[
	0 \to \mathbb Z \stackrel{j}\to \Pic \Xbar \to \Pic \Ubar \to 0,
	\]
	where $j$ maps~$n$ to~$-nK_X$.  This shows that~$\Pic \Ubar \cong \Pic \Xbar/\mathbb Z C \cong \mathbb Z^4$, since the anticanonical divisor class~$-K_X=3L_0-L_1-L_2-L_3-L_4$ is primitive.  So~$\Pic \Ubar$ is torsion free and from the inflation--restriction sequence we conclude that the inflation homomorphism induces an isomorphism
	\[
	\H^1(\Gal(K/k),\Pic U_K) \xrightarrow{\text{inf}} \H^1(G_k,\Pic \Ubar).
	\]
	
	We will compute the action of $\sigma$ on the quotient $\Pic \Ubar$ of $\Pic \Xbar$ using the specific action of $\sigma$ on $\Pic \Xbar$ in Lemma~\ref{lem:anticandivsondp5}.  The classes~$[L_0]$,~$[L_1]$,~$[L_2]$ and~$[L_3]$ in~$\Pic U_K$ form a basis and in this basis the class of~$L_4$ becomes~$[L_4]=3[L_0]-[L_1]-[L_2]-[L_3]$.  So~$\sigma$ acts on~$\Pic \Ubar$ as
	\[
	\sigma = \left(
	\begin{array}{cccc}
	2 & 1 & 1 & 3\\
	-1 & -1 & 0 & -1\\
	-1 & -1 & -1 & -1\\
	-1 & 0 & -1 & -1
	\end{array}
	\right)
	\]
	By results on group cohomology of cyclic groups \cite[Theorem~6.2.2]{Weibel} we get
	\[
	\H^1(G,\Pic \Ubar) \cong \ker(1+\sigma+\sigma^2+\sigma^3+\sigma^4)/\Im(1-\sigma).
	\]
	Since~$1+\sigma+\sigma^2+\sigma^3+\sigma^4 = 0$ and the image of~$1-\sigma$
	is generated by~$(1,0,0,2)$, $(0,1,0,4)$, $(0,0,1,4)$ and~$(0,0,0,5)$ we find
	\[
	\Br_1 U/\Br k \cong \mathbb Z/5\mathbb Z.\qedhere
	\]
\end{proof}

Consider an interesting log K3 surface $U$.  On the compactification $X$ of $U$ we have three important effective anticanonical divisors.  First of all $C=X\backslash U$, but also the two divisors supported on $-1$-curves as described in Lemma~\ref{lem:anticandivsondp5}. These anticanonical sections are important enough to introduce some notation.

\begin{dfn}\label{dfn:l1andl2}
	Let $X$ be an interesting del Pezzo surface of degree $5$ over a field $k$.  Let $l_1,l_2 \in \H^0(X,\omega_X^\vee)$ be the anticanonical sections supported on $-1$-curves from Lemma~\ref{lem:anticandivsondp5}.
\end{dfn}

We will use these elements to construct explicit generators of~$\Br_1 U/\Br k$ for an interesting log K3 surface $U=X\backslash C$.

\begin{lem}\label{lem:genofalgbrgrp}
	Let $K$ be the splitting field of an interesting log K3 surface $U=X \backslash C$ over a number field~$k$. Fix a generator~$\sigma$ of $\Gal(K/k)\cong \mathbb Z/5\mathbb Z$.  Let $h \in \H^0(X,\omega_X^\vee)$ be a global section whose divisor of zeroes is $C$.
	
	The cyclic $\kappa(X)$-algebras
	\[
	\left(\frac{l_1}h, \sigma\right) \quad \text{ and } \quad \left(\frac{l_2}h, \sigma\right)
	\]
	are similar over $\kappa(X)$, their class lies in the subgroup~$\Br U \subseteq \Br \kappa(X)$ and generates~$\Br_1 U/\Br k$.
\end{lem}

\begin{proof}
	As~$\div_U(\frac{l_1}h)$ and~~$\div_U(\frac{l_2}h)$ are orbits of~$-1$-curves defined over~$K$ it follows from Lemma~\ref{lem:cyclicalgebraoverscheme} that the cyclic algebras lie in the subgroup $\Br U$.  The algebras $\left(\frac{l_1}h, \sigma\right) \otimes \left(\frac{l_2}h, \sigma\right)\opp$ and $\left(\frac{l_1}{l_2}, \sigma\right)$ are similar and $\div_U(\frac{l_1}{l_2})$ is the norm of a principal divisor on $U$ since this is even the case on $X$.  Indeed, the divisors $L_{14}+L_1-L_2$ and $L_{24}$ are linearly equivalent on $X$, and their norms $\Norm_{K/k}(L_{14}+L_1-L_2)$ and $\Norm_{K/k}(L_{24})$ are the divisors of zeroes of $l_1$ and $l_2$. It follows again from Lemma~\ref{lem:cyclicalgebraoverscheme} that $\left(\frac{l_1}{l_2}, \sigma\right)$ is trivial in~$\Br U$.
	
	The algebra $\mathcal A$ is split by the degree $5$ extension~$K$ and this implies that $\mathcal A$ is either trivial or of order~$5$.  Suppose that the class of $\mathcal A$ is trivial, then by Lemma~\ref{lem:cyclicalgebraoverscheme} there is a principal divisor $D$ on~$U_K$ such that $\Norm_{K/k} D = \div_U l_1$. This implies that there is a $g\in \kappa(U_K)$ such that $\div_{U_K} g=D$.  Consider $g$ as a function on $X_K$ and $D$ as a divisor on $X_K$.  Then $\div_{X_K} g=D+nC$ for some non-negative integer $n$, since $C$ is geometrically irreducible. From $\Norm_{K/k} D = \div_U l_1$ we find $K_{X_K}\cdot D = -1$ and we conclude that
	\[
	0=K_{X_K} \cdot \div_{X_K} g=K_{X_K}\cdot D+nK_{X_K} \cdot C=-1+5n,
	\]
	which is a contradiction.  
\end{proof}

Note that $l_1$ and $l_2$ are only defined up to multiplication by an element in $k^\times$.  From now on we will denote the class in Lemma~\ref{lem:genofalgbrgrp} by~$\mathcal A \in \Br_1(U)$ which is uniquely defined up to an element in $\Br k$.  Fix for the moment an interesting del Pezzo surface $X$.  We will consider the class $\mathcal A_h$ on $U_h$ as $h$ varies over all hyperplane sections.  We have seen that $\mathcal A_h$ is of order~$5$ if~$h$ cuts out a geometrically irreducible curve.  The next lemma shows that this only fails for specific choices of~$h$.

\begin{lem}\label{lem:geomirrdiv}
	Let~$X \subseteq \mathbb P^5_k$ be an interesting del Pezzo surface over a field~$k$.  A hyperplane section given by the vanishing of an~$h\in \H^0(X,\mathcal O(1))$ fails to be geometrically irreducible if and only if $h$ is a scalar multiple of either~$l_1$ or~$l_2$.
\end{lem}

\begin{proof}
	Consider a hyperplane section~$C \subseteq X$.  Let~$D$ be a $k$-irreducible component of $C$ and consider a $-1$-curve~$L$ on $X\sep$.  It follows that~$L \cdot D\sep=\sigma(L) \cdot D\sep$ and as the Galois orbit of~$L$ is an anticanonical divisor, we find
	\[
	5 \geq -K_X \cdot D = 5L \cdot D\sep >0,
	\]
	since the degree of $D \subseteq \mathbb P^5_k$ is positive and at most the degree of $C$, which equals~$5$.  This proves that $L \cdot D\sep=1$ for all $-1$-curves $L$ and hence $C-D$ is an effective divisor of degree~$0$.  We conclude that $C=D$ and this proves that any anticanonical section $C$ is irreducible over $k$.
	
	If $C$ is not geometrically irreducible, then it must have at least two geometrically irreducible components of the same degree $d$ since the Galois group acts on the set of geometrically irreducible components of $C$.  Since $C$ is of degree $5$ we find $2d \leq 5$ and hence $d$ is either $1$ or $2$.  But in both cases we see that $C$ contains a geometrically irreducible curve of degree $1$, which must be a $-1$-curve $L$.  Then $C$ also contains all conjugates of $L$ and hence $C$ is the Galois orbit of a $-1$-curve.  This proves that $C$ is defined by the vanishing of either $l_1$ or $l_2$.
\end{proof}

So we have seen that a log K3 surface $U=X\backslash C$ of $\dP5$ type over a number field $k$ with $C$ geometrically integral has a non-trivial algebraic Brauer group modulo constants precisely for one action of Galois on the lines. We can use the correspondence in Proposition~\ref{prop:bijectionisomdp5s} to classify interesting del Pezzo surfaces over a field $k$.

\begin{prop}
	Let $k$ be a field with a fixed separable closure $k\sep$.  The map which sends an isomorphism class of interesting del Pezzo surfaces over $k$ to its splitting field $K \subseteq k\sep$ is a bijection to the set of degree $5$ Galois extensions of $k$ contained in $k\sep$.
\end{prop}

\begin{dfn}
	Let $K/k$ be a Galois extension of degree $5$.  The isomorphism class of interesting del Pezzo surfaces of degree $5$ over $k$ which are split by $K$ is denoted by $\dP5(K)$.
\end{dfn}

To obtain equations for an interesting del Pezzo surface given its splitting field we can use Proposition~\ref{prop:constructiondp5s}. We can also use this to recover the anticanonical sections $l_1$ and $l_2$; in the notation of that proposition, let $\Lambda_{i,j}$ be the line through the points $P_i$ and $P_j$, where we consider the indexes modulo $5$. The divisors $\sum \Lambda_{i,i+1}$ and $\sum \Lambda_{i,i+2}$ are defined over $k$ and lie in the linear system $\mathcal L$. These are the only divisors in $\mathcal L$ supported on lines and correspond to $l_1$ and $l_2$ on $X$.

\section{A model of $\dP5(\mathbb Q(\zeta_{11})^+)$ over the integers}

We have seen that all interesting del Pezzo surfaces split by a specific quintic extension $K$ of the base field $k$ are isomorphic. We have also seen how to construct such a surface as the image of a rational map $\mathbb P_k^2 \dashrightarrow \mathbb P_k^5$. We will now give an explicit first example of how one can construct models of this surface.

We will use the quintic extension $K=\mathbb Q(\alpha)$ of $k=\mathbb Q$ where $\alpha = \zeta_{11}+\zeta_{11}^{-1}$. We will write $m_\alpha$ for the minimal polynomial of $\alpha$ over $\mathbb Q$. Let $\alpha_i$ be the conjugates of $\alpha$.

\begin{dfn}
Let $\mathcal Q \subseteq \mathbb Z[x,y,z]_{(5)}$ be the sub-$\mathbb Z$-module consisting of all quintic polynomials which vanish at least twice at the points $P_i = (\alpha_i^2 \colon \alpha_i \colon 1) \in \mathbb P^2_{\mathbb Q}$.
\end{dfn}

\begin{lem}
The $\mathbb Z$-module $\mathcal Q$ is free of rank $6$ and $\mathbb Z[x,y,z]_{(5)}/\mathcal Q$ is torsion free.
\end{lem}

\begin{proof}
Clearly $\mathcal Q$ is a free $\mathbb Z$-module. To compute its rank we use the results in \cite{dP5s} which says that $\mathcal Q \otimes \mathbb Q[x,y,z]$ has dimension $6$. The last statement follows from the fact that for $\lambda \in \mathbb Z\backslash \{0\}$ and $q \in \mathbb Z[x,y,z]_{(5)}$ we have $\lambda q \in \mathcal Q$ precisely if $q \in \mathcal Q$.
\end{proof}

Let us fix a basis $q_i \in \mathcal Q$.

\begin{dfn}
Let $\mathcal X \subseteq \mathbb P^5_{\mathbb Z}$ be the image of the rational map
$\mathbb P^2_{\mathbb Z} \dashrightarrow \mathbb P^5_{\mathbb Z}$ defined by the $q_i$.

There are two primitive elements of $\mathcal Q$ which factor into linear polynomials over $\bar{\mathbb Q}$. These correspond to the two primitive linear forms $l_1,l_2 \in \mathbb Z[u_0,u_1,\ldots, u_5]$. 
\end{dfn}

Note that the scheme $\mathcal X$ does not depend on the choice of basis of $\mathcal Q$. It does however depend on the choice of $\alpha$. The statements are easier, but not by much, since we have chosen an integral $\alpha$; we could have picked any generator of $K$ over $\mathbb Q$.

\begin{prop}\label{prop:equationsXXdP511}
The scheme $\mathcal X \subseteq \mathbb P^5_{\mathbb Z}$ is given by the equations\\[-.6cm]
\begin{multline*}
u_0u_3+22u_0u_4+121u_0u_5-u_1^2-121u_1u_3+2662u_1u_4-36355u_2u_4-\\
9306u_2u_5+10494u_3u_4-242u_3u_5-215501u_4^2+68123u_4u_5-13794u_5^2,
\end{multline*}\\[-1.2cm]
\begin{multline*}
u_0u_4+11u_0u_5-u_1u_2-11u_1u_3+242u_1u_4-3223u_2u_4-847u_2u_5+\\
902u_3u_4-11u_3u_5-19272u_4^2+6413u_4u_5-1331u_5^2,
\end{multline*}\\[-1.2cm]
\begin{multline*}
u_0u_5-u_1u_3+22u_1u_4-u_2^2-286u_2u_4-77u_2u_5+77u_3u_4-
1694u_4^2+572u_4u_5-121u_5^2,
\end{multline*}\\[-1.2cm]
\begin{multline*}
u_1u_4-u_2u_3-11u_2u_4-77u_4^2+55u_4u_5-11u_5^2,\\
\end{multline*}\\[-1.7cm]
\begin{multline*}
u_1u_5-u_2u_4-11u_2u_5-u_3^2+11u_3u_4-44u_4^2.\\
\end{multline*}\\[-0.8cm]
In this example, the two relevant hyperplane sections are given by
\[
l_1 = u_0 + 22u_1 - 363u_2 + 165u_3 - 1859u_4 + 484u_5,
\]
\[
l_2 = u_0 + 22u_1 - 352u_2 + 143u_3 - 1595u_4 + 363u_5.
\]
Also,
\begin{enumerate}
\item the generic fibre $X = \mathcal X_{\mathbb Q}$ is isomorphic to $\dP5(K)$, and
\item $\mathcal X$ is the flat closure of $X$ in $\mathbb P^5_{\mathbb Z}$.
\end{enumerate}
\end{prop}

\begin{proof}
The \textsc{magma} code for these computations can be found in \cite{Lyczakcode}. We use those computation also for the proofs of some of the following statements.
\begin{enumerate}
\item This follows from the fact that $\dP5(K)$ is the image of the rational map $\mathbb P^2_{\mathbb Q} \dashrightarrow \mathbb P^5_{\mathbb Q}$ using a basis of $\mathcal Q \otimes \mathbb Q$.
\item We used a Gr\"obner basis computation to compute the image of the rational map $\mathbb P^2_{\mathbb Z} \dashrightarrow \mathbb P^5_{\mathbb Z}$. The upshot of this that the equations above also define a Gr\"obner basis of the ideal $I \subseteq \mathbb Z[u_0,u_1,\ldots,u_5]$ of $\mathcal X \subseteq \mathbb P^5_{\mathbb Z}$. Since the leading coefficients of the basis elements are monic we conclude from \cite[Proposition~4.4.4]{AdamsLoustaunauGroebnerBases} that $I\mathbb Q[u_0,u_1,\ldots, u_5] \cap \mathbb Z[u_0,u_1,\ldots, u_5]$ is equal to $I$ and hence that $\mathcal X$ is the flat closure of its generic fibre. \qedhere
\end{enumerate}
\end{proof}

This last proof also shows that $\mathcal X$ itself is integral, since $X$ is integral. From this or the fact that $\mathcal X$ is flat over $\mathbb Z$ we deduce the important fact that all fibres $\mathcal X_\ell$ are equidimensional of dimension $2$.

\subsection{Fibres of the model}

We can now study almost all fibres of $\mathcal X \to \Spec(\mathbb Z)$ using the reduction of the minimal polynomial $m_\alpha$ modulo primes.

\begin{lem}\label{lem:smoothfibres}
Let $\ell \in \mathbb Z$ be a prime for which the reduction $\bar m_\alpha \in \mathbb F_\ell[s]$ is separable. The fibre $\mathcal X_\ell$ is a del Pezzo surface of degree $5$.
The hyperplane section given by the vanishing of $l_1$ and $l_2$ each cut out five $-1$-curves on $\mathcal X_\ell$.
\end{lem}

\begin{proof}
Consider the rational map $\mathbb P^2_{\mathbb F_\ell} \dashrightarrow \mathbb P^5_{\mathbb F_\ell}$ using the basis $q_i\otimes 1$ of $\mathcal Q \otimes \mathbb F_\ell$.  By construction this morphism lands in $\mathcal X$.

If $m_\alpha$ is separable modulo $\ell$ then the reductions of the points $P_i$ modulo $\ell$ are distinct and $\mathcal Q \otimes \mathbb F_\ell$ consists of all quintics over $\mathbb F_\ell$ vanishing at least twice at these points. By Proposition~\ref{prop:constructiondp5s} we see that the image $Y$ is a del Pezzo surface of degree $5$. Hence we have $Y \subseteq \mathcal X_\ell \subseteq \mathbb P^5_{\mathbb F_\ell}$.

From Corollary~III.9.6 in \cite{Hartshorne} we see that all irreducible components of $\mathcal X_\ell$ are of dimension $2$. Hence $Y$ is one such component of $\mathcal X_\ell$. By flatness $\mathcal X \to \Spec \mathbb Z$ we see that $\mathcal X_\ell$ has degree $5$ in $\mathbb P^5_{\mathbb F_\ell}$, just like $Y$. Hence $\mathcal X_\ell$ has no other irreducible components.

The statement about $l_1$ and $l_2$ also follows from the flatness of $\mathcal X$ over $\mathbb Z$.
\end{proof}

There are actually two possibilities if the reduction of $m_\alpha$ modulo $\ell$ is separable.

\begin{cor}\label{cor:splitandinterestingfibres}
\begin{enumerate}
\item If $m_\alpha$ is irreducible modulo $\ell$ then $\mathcal X_\ell$ is interesting.
\item If $m_\alpha$ splits completely in $\mathbb F_\ell$ with distinct roots, then $\mathcal X_\ell$ is split, i.e.~all $-1$-curves are defined over $\mathbb F_\ell$.
\end{enumerate}
\end{cor}

\begin{proof}
In the proof of the previous lemma we have seen that $\mathcal X_\ell$ is the image of $\mathbb P^2_{\mathbb F_\ell}$ of all quintics vanishing at least twice at the five points $P_i=(\alpha_i^2 \colon \alpha_i \colon 1)$ modulo $\ell$. The action of Galois on the $-1$-curves on $\mathcal X_\ell$ is determined by the action of Galois on the points $P_i$.

We have seen before that an interesting del Pezzo surface is obtained precisely if the five points are defined over a quintic extension, but are conjugate over the base field. A split del Pezzo surface of degree $5$ corresponds to the case that all points are defined over the base field. 
\end{proof}

\begin{remark}
It is even possible to determine the fibres of $\mathcal X/\mathbb Z$ directly from the splitting of $m_\alpha$ in $\mathbb F_\ell$ even for primes which divide the discriminant of $m_\alpha$. But this requires a long geometric treatise of singular del Pezzo surfaces, see the Ph.D.~thesis of the author \cite{Lyczakthesis}. For our explicit examples it is much shorter to just study the remaining finitely many fibres separately.
\end{remark}

For this example we are only left with the fibre over $\ell = 11$, since $\Delta(m_\alpha)=11^4$.

\begin{lem}\label{lem:singularfibreat11}
The fibre $\mathcal X_{11}$ is an integral surface in $\mathbb P^5_{\mathbb F_{11}}$ with precisely one singular point.

The divisor on $\mathcal X_{11}$ defined by $l_1$ is supported on a line $L$. The singular point lies on this line. Also, a hyperplane section given by $h \in \mathbb F_{11}[u_0,u_1,\ldots,u_5]$ contains $L$ precisely if $h$ lies in $\mathbb F_{11}[u_0,u_1,\ldots,u_4]$.
\end{lem}

\begin{proof}
We have explicit equations for $\mathcal X$ and hence for $\mathcal X_{11}$ and all statements can be checked explicitly, for example by \textsc{magma} \cite{Lyczakcode}.
\end{proof}

The surface $\mathcal X_{11}$ is actually well-understood. It is a singular del Pezzo surface and the unique singular point which is of type A$_4$ and lies on a unique line $L$ on $\mathcal X_{11} \subseteq \mathbb P^2_{\mathbb F_{11}}$, i.e.~a $-1$-curve (on its minimal desingularisation).

We even have a birational morphism $\mathbb P^2 \dashrightarrow \mathcal X_{11}$ which restricts to an isomorphism $\mathbb A^2 \xrightarrow{\cong} \mathcal X_{11}\backslash L,$
\[
(1\colon y \colon z) \mapsto (1 \colon y \colon z \colon y^2 \colon yz \colon y^3 + z^2).
\]

This will allow one to transfer many problems on $\mathcal X_{11}$ to a problem on the affine or the projective plane.

\begin{remark}
This is not at all particular to this one example; for any choice $\alpha \in \bar{\mathbb Q}$ of degree $5$ we can construct a relative surface $\mathcal X$ over $\mathbb Z$. If the minimal polynomial $m_\alpha$ reduces to the fifth power of a linear polynomial modulo $\ell$ then $\mathcal X_\ell$ always has these properties.

We will forgo this general approach and stick to our explicit examples.
\end{remark}

\section{A family of log K3 surfaces}

Consider the model $\mathcal X \subseteq \mathbb P^5_{\mathbb Z}$ of an interesting del Pezzo surface of the previous section. We will use it to construct a family of log K3 surfaces of $\dP5$ type together with their models.

\begin{dfn}
Let $h \in \mathbb Z[u_0,u_1,u_2,u_3,u_4,u_5]_{(1)}$ be a primitive linear form. Let $\mathcal U_h$ be the complement of $\mathcal C_h=\{h=0\} \cap \mathcal X$ in $\mathcal X$.
\end{dfn}

We will consider when $\mathcal U_h$ does not have integral points. First of all this happens when $\mathcal U_h$ is not everywhere locally soluble. We can make precise when this happens.

\begin{lem}\label{lem:localsolubilitydp511}
	The affine surface~$\mathcal U_h$ is everywhere locally soluble precisely when
	\[
	h \not \equiv u_2+u_5 \mod 2.
	\]
\end{lem}

\begin{proof}
One can check that the points
\begin{center}
\small
\begin{tabular}{ccc}
$(1 \colon 0 \colon 0 \colon 0 \colon 0 \colon 0)$ & $(-693 \colon -88 \colon -11 \colon 0 \colon 1 \colon 1)$ & $(-725 \colon -120 \colon -11 \colon 1 \colon 0 \colon 1)$ \\
$(967 \colon 122 \colon 11 \colon -1 \colon 0 \colon 1)$ & $(-3345 \colon -328 \colon -46 \colon -4 \colon 4 \colon 4)$ & $(-3497 \colon -331 \colon -34 \colon 1 \colon 1 \colon 0)$\\
$(-6138 \colon -407 \colon -44 \colon 0 \colon 1 \colon 0)$
\end{tabular}
\end{center}
lie on $\mathcal X$. Also, their coordinates as vectors in $\mathbb Z^6$ define a lattice of dimension $6$ of index $2$. This proves that for any prime $\ell \ne 2$ and any hyperplane section $h$ at least one of these points $P$ satisfies $h(P) \not \equiv 0 \mod p$.  This shows that such a point determines an element in $\mathcal U_h(\mathbb Z_\ell)$.

We have seen in Lemma~\ref{lem:smoothfibres} that $\mathcal X_2$ is smooth. One can check that $\#\mathcal X(\mathbb F_2)=5$ and that these points lie on the indicated hyperplane over $\mathbb F_2$.
\end{proof}

\subsection{Obstructions coming from $\mathcal A_h$}

Note that if $\mathcal C_\mathbb Q$ is geometrically irreducible, i.e.~$h$ is not a multiple of $l_1$ and $l_2$ by Lemma~\ref{lem:geomirrdiv}, then we see that $\Br U_h/\Br \mathbb Q$ contains an element of order $5$. Let us compute the invariant maps for this element.

\begin{lem}\label{lem:effevaldp511}
	Consider a geometrically irreducible hyperplane section given by a primitive~$h$.  Let~$\ell$ be a prime and let~$\mathcal A$ be a generator for~$\Br_1 U_h/\Br \mathbb Q$.  We consider the invariant map
	\[
	\inv_\ell \mathcal A \colon \mathcal U_h(\mathbb Z_\ell) \to \mathbb Q/\mathbb Z.
	\]
If $\ell \ne 11$ then $\inv_\ell \mathcal A$ is identically zero.
\end{lem}

\begin{proof}
The statement is immediate for the infinite place and primes $\ell$ which split completely in $K$. Since in those cases $\mathcal A_{\mathbb Q_\ell} \cong \mathcal A_{K_{\mathfrak l}}$ is trivial in $\Br U_{\mathbb Q_\ell}$ for any prime $\mathfrak l$ of $K$ above $\ell$.

Now suppose that $m_\alpha$ is irreducible modulo $\ell$. The Kummer--Dedekind theorem implies that $\ell$ is inert in $\mathbb Z[\alpha]$. This also proves that $\ell$ is inert in $\mathcal O_K$ and there is a unique prime $\mathfrak l$ above $\ell$. Also, $U_{\mathbb Q_\ell}$ is an interesting del Pezzo surface since $m_\alpha$ is irreducible over $\mathbb Q_\ell$. Hence the hyperplane section given by the vanishing of $l_1$ modulo $\ell$ is geometrically irreducible and does not contain $\mathbb F_\ell$-points. Hence $l_1$ is invertible on all points in $U(\mathbb Z_\ell)$. This shows that $\frac{l_1}h(P) \in \mathbb Z_\ell^\times$ for all $P \in \mathcal U(\mathbb Z_\ell)$. Since the extension $K_{\mathfrak l}/\mathbb Q_\ell$ of local fields is unramified we see that any unit is a norm. Hence $\inv_\ell \mathcal A$ is also in this case constantly $0$.
\end{proof}

\begin{lem}
Let $L$ be the unique line on $\mathcal X_{11} \subseteq \mathbb P^5_{\mathbb F_{11}}$. If $L$ does not lie in the zero locus of $h$ then $\inv_{11} \mathcal A \colon \mathcal U_h(\mathbb Z_{11}) \to \frac15\mathbb Z/\mathbb Z$ is surjective.
\end{lem}

\begin{proof}
We have seen that in Lemma~\ref{lem:singularfibreat11} that the condition is equivalent to $h \mod 11$ being dependent on $u_5$.

On points $P$ for which $\frac{l_1}h(P) \in \mathbb Z_{11}$ is invertible and we can use Lemma~\ref{lem:trivialcycalgs} to compute the $\inv_{11} \mathcal A(P)$, i.e.~the invariant at $P$ only depends on $\frac{l_1}h(P) \in \mathbb F_{11}$ up to fifth powers and there is an isomorphism of $\psi \colon \mathbb F_{11}/\{\pm 1\} \to \frac15 \mathbb Z/\mathbb Z$ such that $\inv_{11} \mathcal A(P)=\psi\left(\frac{l_1}h(P)\right)$.

Hence it will suffice to prove the following stronger statement: the map $\frac{l_1}h \colon \left(\mathcal U_h\backslash L\right)(\mathbb F_{11}) \to \mathbb F_{11}^\times$ is surjective. Note that both the domain and the map depend on our choice of $h$. For this statement we only have finitely many $\bar h \in \mathbb F_{11}[u_0,u_1,\ldots,u_5]_{(1)}$ which we need to evaluate on a subset of the finitely many points $\mathcal X_{11}(\mathbb F_{11})$. The code for this computation can be found in \cite{Lyczakcode}.
\end{proof}

\begin{prop}\label{prop:invmapsingularfibre11}
Define $f = h(1, y, z, y^2, yz, y^3 + z^2) \in \mathbb F_{11}[y,z]$. The value~$0\in \mathbb Q/\mathbb Z$ lies in the image of~$\inv_{11} \mathcal A$ precisely when the polynomial $f$	assumes a values~$\pm 1$ modulo~$11$ for~$y,z \in \mathbb F_{11}$.

The image of $\inv_{11} \mathcal A \colon \mathcal U_h(\mathbb Z_{11}) \to \frac15\mathbb Z/\mathbb Z$ has size
		\begin{itemize}
			\item $1$ precisely when~$f$ is a constant;
			\item $4$ precisely when~$f$ is a separable quadratic polynomial in $y$;
			\item $5$ in all other cases.
		\end{itemize}
\end{prop}

Note in the second case that $f$ is in particular independent of $z$.

\begin{proof}
Using the last lemma we will only need to consider the $\bar h$ over $\mathbb F_{11}$ which do not depend on $u_5$. In this case we have that
\[
\left(\mathcal U_h\backslash L\right)(\mathbb Z_{11}) = \mathcal U_h(\mathbb Z_{11})
\]
since $L$ lies in the zero locus of $h$. Furthermore, the value of $\inv_{11}\mathcal A$ at a point $P$ only depends on $\frac{l_1}h(P)$ modulo $11$ or equivalently the reduction $\bar P \in \mathcal U_h(\mathbb F_{11})$ of $P$. The statement can now be checked completely by a computer.

We would however like to provide a little more insight. Using the isomorphism $\mathcal X_{11}\backslash L \xrightarrow{\cong} \mathbb A_{\mathbb F_{11}}^2$ from Lemma~\ref{lem:singularfibreat11} we see that $\mathcal U_{h,11}=\mathcal U_{h,11}\backslash L \xrightarrow{\cong} \mathbb A^2_{\mathbb F_{11}}\backslash \{f = 0\}$. Hence we are interested in the image of $f \colon \mathbb A^2_{\mathbb F_{11}}\backslash \{f = 0\} \to \mathbb F_{11}^\times/\{\pm 1\}$. If $\bar h$ depends on $u_5$, then $f$ is a cubic polynomial and $f=c$ for any $c \in \mathbb F_{11}$ is likely to have a solution, as made precise in the previous lemma. If $\bar h$ depends on either $u_2$ or $u_4$, then $f$ is linear in $z$ with the leading coefficient being linear in $y$. Fixing $y$ to be a suitable $y_0$ shows that $f(y_0,z)=c$ always has a solution in $\mathbb F_{11}$.

In the remaining case $f$ is a polynomial independent of $z$ of degree at most $2$. When $f$ is constant we immediately get the first case. Whenever $f$ is linear or a inseparable quadratic polynomial with root $\rho \in \mathbb F_{11}$ the surjectivity of $f \colon \mathbb F_{11}\backslash\{\rho\} \to \mathbb F_{11}^\times/\{\pm 1\}$ is immediate.

For the last case it is easily checked that for a quadratic separable polynomial $f=c(y-\rho_1)(y-\rho_2)$ the image of $f \colon \mathbb F_{11}\backslash \{\rho_1,\rho_2\} \to \mathbb F_{11}^\times/\{\pm 1\}$ has size four. This is independent of whether $f$ splits over $\mathbb F_{11}$ or over $\mathbb F_{11^2}$.
\end{proof}

We can now apply the above results to compute the Brauer--Manin obstruction for a fixed~$h$ and find actual algebraic obstructions of order~$5$ to the integral Hasse principle.

\begin{thm}\label{thm:algebraicobstructions}
	Let~$\mathcal H$ be the hyperplane in~$\mathbb P^5_\mathbb Z$ given by the vanishing of~$u_1-6u_3$.  The complement~$\mathcal U = \mathcal X \backslash \mathcal H$ has points over~$\mathbb Q$ and every~$\mathbb Z_\ell$, but there is an order $5$ Brauer--Manin obstruction to the existence of integral points.
\end{thm}

\begin{remark}
Let $S$ be a set of rational primes which split completely in $K$. The proof of the above statement can easily be adapted to show that there are no $S$-integral points on $\mathcal X$.

On the other hand if $\ell$ is an inert prime then $\inv_\ell \mathcal A$ need not be constant on $\mathbb Q_\ell$-points even if it is so on $\mathbb Z_\ell$-points. Although our model $\mathcal U$ is regular this does not contradict Theorem~1 in \cite{effeval}. Hence the concept of a regular model is not as useful for $S$-integral points as it is for rational points.
\end{remark}

A careful analysis of the above proof yields the following result.

\begin{thm}\label{thm:bmobstructionondp511}
	Let~$\mathcal U_h$ be the complement in~$\mathcal X$ of a geometrically irreducible hyperplane section given by a primitive linear form~$h \in \mathbb Z[u_0,u_1,\ldots, u_5]$.  The class of~$h$ modulo~$2$ determines whether the affine surface~$\mathcal U_h$ is locally soluble.  The existence of an algebraic obstruction to the Hasse principle for integral points depends only on the reduction of~$h$ modulo~$11$.  Out of the~$11^6-1=1771560$ possible reductions $\bar h$ of $h$ modulo~$11$ precisely~$228$ give an obstruction.
\end{thm}

Note that this does not mean that the reduction of~$h$ modulo~$2$ and~$11$ is the only condition; the proof still uses the assumption that~$h$ is primitive.  It follows from Lemma~\ref{lem:geomirrdiv} that the condition that the section is geometrically irreducible is immediately satisfied if~$h$ does not reduce to~$\pm u_0$ modulo $11$.  For hyperplanes $h$ reducing to either of these two form it is easily shown that $\inv_{11}\mathcal A$ is identically equal to~$0$ on $\mathcal U_h(\mathbb Z_{11})$.

\begin{proof}[Proof of Theorem~\ref{thm:bmobstructionondp511}]
Let us count the non-zero linear forms~$\bar h$ over~$\mathbb F_{11}$ for which such an obstruction exists.  In Proposition~\ref{prop:invmapsingularfibre11} we saw that we get no obstruction unless~$f$ is either constant or a separable quadratic polynomial in $y$.
	
	If~$f$ is constant then we see that~$\inv_{11} \mathcal A$ is constant and we get an obstruction if~$f$ is one of the~$8$ non-fifth powers modulo~$11$.
	
	For an~$\bar h$ such that $f$ is a quadratic inseparable polynomial we have seen that~$f \colon \mathbb F_{11}\backslash \{\rho_1,\rho_2\} \to \mathbb F_{11}^\times/\{\pm 1\}, x \mapsto f(x)$ misses exactly one value.  If~$f$ misses the value~$q \in \mathbb F_{11}^\times/\{\pm 1\}$ we see that~$\lambda f$ for $\lambda \in \mathbb F_{11}^\times$ misses the class of~$\lambda q$. There are~$10\cdot 11^2$ quadratic polynomials over~$\mathbb F_{11}$ and~$10\cdot 11$ of these are inseparable.  The group~$\mathbb F_{11}^\times$ acts on the remaining~$10^2\cdot 11$ quadratic polynomials by multiplication.  All orbits have size~$10$ and in such an orbit exactly~$2$ miss the unit element in~$\mathbb F_{11}^\times/(\mathbb F_{11}^\times)^5$.  This proves that for an~$h$ for which the invariant map at $11$ assumes precisely $4$ values there is an obstruction if the associated polynomial~$f$ is one of these~$2\cdot 10\cdot 11=220$ separable quadratic polynomials.
\end{proof}

\begin{remark}
We can use the same construction to produce models $\mathcal X$ with a different splitting field $K$. However if $K$ is ramified at a prime $p>11$ then one can show that the image of $\inv_p \mathcal U_h(\mathbb Z_p)$ has either size $1$ on $5$. Furthermore, the first case happens precisely when $\frac{l_1}h$ is constant modulo $p$ similar to above. The interesting thing to note is that the intermediate case in which the invariant map assumes $4$ invariants does not occur any more.
\end{remark}

We are left with the case of quintic fields $K/\mathbb Q$ which are ramified at $5$.

\section{Explicit examples with splitting field $K\subseteq \mathbb Q(\zeta_{25})$}

It is also possible to find obstructions of order~$5$ to the integral Hasse principle when~$\mathcal X$ is a model of the interesting del Pezzo surface $X$ split by the unique quintic extension $K$ contained in $\mathbb Q(\zeta_{25})$.  In that case~$K$ has~$5$ as a wildly ramified prime.  For example, define the field~$K \subseteq \mathbb Q(\zeta_{25})$ as the splitting field of the polynomial
\[
m_{\alpha}=s^5 - 20s^4 + 100s^3 - 125s^2 + 50s - 5.
\]
This produces the projective surface~$\mathcal X$ over the integers given by the five equations
\begin{multline*}
u_0u_3+40u_0u_4+400u_0u_5-u_1^2-400u_1u_3+16000u_1u_4-365050u_2u_4
-\\
49995u_2u_5+51985u_3u_4-200u_3u_5-2029975u_4^2+392250u_4u_5-39375u_5^2,
\end{multline*}\\[-1.2cm]
\begin{multline*}
u_0u_4+20u_0u_5-u_1u_2-20u_1u_3+800u_1u_4-18125u_2u_4-2500u_2u_5+\\
2550u_3u_4-5u_3u_5-101015u_4^2+19800u_4u_5-2000u_5^2,
\end{multline*}\\[-1.2cm]
\begin{multline*}
u_0u_5-u_1u_3+40u_1u_4-u_2^2-900u_2u_4-125u_2u_5+125u_3u_4-5000u_4^2+985u_4u_5-100u_5^2,\\
\end{multline*}\\[-1.7cm]
\begin{multline*}
u_1u_4-u_2u_3-20u_2u_4-125u_4^2+50u_4u_5-5u_5^2,\\
\end{multline*}\\[-1.7cm]
\begin{multline*}
u_1u_5-u_2u_4-20u_2u_5-u_3^2+20u_3u_4-100u_4^2.\\
\end{multline*}\\[-0.8cm]
The two hyperplane sections over~$\mathbb Z$ cutting out the two quintuples of~$-1$-curves are
\[
l_1 = u_0 + 25u_1 - 700u_2 + 200u_3 - 3425u_4 + 575u_5,
\]
\[
l_2 = u_0 + 75u_1 - 1675u_2 + 375u_3 - 5175u_4 + 575u_5.
\]
By construction this scheme shares many properties with the previous example.

\begin{prop}
\begin{enumerate}
\item The scheme $\mathcal X$ is integral, with integral fibres.
\item If $m_\alpha$ is irreducible modulo $\ell$ then $\mathcal X_\ell$ is an interesting del Pezzo surface.
\item If $m_\alpha$ has five distinct roots in $\mathbb F_\ell$ then $\mathcal X_\ell$ is a split del Pezzo surface of degree $5$.
\item For $\ell=5$ the surface $\mathcal X_5$ contains a unique line $L$, has a unique singular point of type A$_4$ which lies on $L$. There is a birational map $\mathcal X_5 \dashrightarrow \mathbb P^2_{\mathbb F_5}$ which restricts to an isomorphism $\mathcal X_5 \backslash L \xrightarrow{\cong} \mathbb A^2_{\mathbb F_5}$.
\end{enumerate}
\end{prop}

The only fibre not discussed in this lemma is the one over $7$. Although $\mathcal X_7$ is again a singular del Pezzo surface, we will not need any information about this fibre this since $7$ splits completely in $K$.

\begin{proof}
One can follow the proofs in Section~3 for this different choice of $\alpha$ and corresponding equations for $\mathcal X$.
\end{proof}

We will consider $\mathcal U_h = \mathcal X\backslash \{h=0\}$ like in the previous sections. As before, local solubility is immediate at most primes.

\begin{lem}\label{lem:localsolubilitydp525}
	The surface~$\mathcal U_h$ is everywhere locally soluble precisely when
	\[
	h \not \equiv u_2+u_3 \mod 2.
	\]
\end{lem}

\begin{proof}
As for the proof of Lemma~\ref{lem:localsolubilitydp511} it is easy enough to find enough points on $\mathcal X$ whose reductions do not lie on a hyperplane modulo $\ell > 2$. For $\ell=2$ the fibre $\mathcal X$ is again smooth, $\#\mathcal X(\mathbb F_2)=5$ and all $\mathbb F_2$-points lie on the unique hyperplane given by $u_2+u_3 \equiv 0 \mod 2$.
\end{proof}

The computation of the invariant maps at the unramified primes is the same computation as in Lemma~\ref{lem:effevaldp511} for the previous example.

\begin{lem}
	Consider the invariant map
	\[
	\inv_\ell \mathcal A \colon \mathcal U_h(\mathbb Z_\ell) \to \mathbb Q/\mathbb Z.
	\]
If~$\ell\ne 5$, then the invariant map is identically zero.
\end{lem}

Let us consider the remaining prime.

\begin{thm}\label{thm:invariantmapsfordp525}
Then $\inv_5 \mathcal A \colon \mathcal U(\mathbb Z_5) \to \frac15 \mathbb Z/\mathbb Z$ is not surjective precisely when there exist integers~$\lambda$,~$c_1$ and~$c_3$ satisfying $5 \nmid \lambda$, and $5\mid c_1,c_3$ or~$5 \nmid c_1$ such that
	\[
	h \equiv \lambda u_0+5(c_1u_1+c_3u_3) \mod 25.
	\]
	The invariant map is constant when~$5 \mid c_1, c_3$ and otherwise the size of its image is~$3$.

The value~$0\in \mathbb Q/\mathbb Z$ lies in the image of~$\inv_{11} \mathcal A$ precisely when~$\lambda+5(c_1 y + c_3 y^2)$ assumes one of values~$\pm 1, \pm 7$ modulo~$25$ for~$y\in \mathbb Z$.
\end{thm}

To prove this result one can use the fact that the model $\mathcal X$ is regular. However, if $5$ is ramified in $K$ one has a similar statement for any model $\mathcal X_\alpha$, which does not need to be regular. The following chain of results also implies a similar result in the more general case.

\begin{lem}
Consider a point $\bar P \in \left(\mathcal U_h\backslash L\right)(\mathbb F_5)$. Let $\mathcal P$ be the set of the $25$ lifts of $\bar P$ in $\mathcal X(\mathbb Z/25\mathbb Z)$. The image of
\[
\frac{l_1}h \colon \mathcal P \to \left(\mathbb Z/25\mathbb Z\right)^\times
\]
is either of size $1$ of $5$.
\end{lem}

\begin{proof}
Define $\mathcal V:=\mathcal U_5\backslash L\subseteq \mathbb A^5_{\mathbb F_5}$ on which $\frac{h}{l_1}$ is given by $h_{\text{aff}}=a_0+a_1u_1+\ldots+a_5u_5$. Now let~$\vec x = (x_1,x_2,x_3,x_4,x_5)$ be a~$5$-tuple of integers reducing to~$\bar P \in \mathcal V$.  We will first show that the lifts of~$\bar P$ to points in~$\mathcal X(\mathbb Z/25 \mathbb Z)$ are~$\vec x+5\vec w$ where~$\vec w$ is any vector in a translation of the tangent space of~$\mathcal V$ at~$\bar P$.
	
	Suppose that $\mathcal X$ is given by polynomials $g_j$ in the variables $u_i$.  The tangent space at $\bar P$ is by definition
	\[
	T_{\bar P}\mathcal V = \left\{\vec v\in \mathbb F_5^5\ \colon\ \sum_{i=1}^5 \frac{dg_j}{du_i}(\bar P) v_i \equiv 0 \mod 5\right\}
	\]
	and if~$\vec x+5\vec w \in \mathcal X(\mathbb Z/25\mathbb Z)$ then for all~$j$
	\[
	0 \equiv g_j(\vec x+5\vec w)\equiv g_j(\vec x)+5\sum_{i=1}^5 \frac{dg_j}{du_i}(\vec x) w_i \mod 25
	\]
	which proves the claim.

To compute~$h_{\text{aff}}$ at these lifts, let us write~$\vec a=(a_1,a_2,a_3,a_4,a_5) \in \mathbb Z^5$.  Then we find
	\[
	h_{\text{aff}}(\vec x +5\vec w) \equiv h_{\text{aff}}(\vec x)+5\vec a\cdot \vec w \mod 25.
	\]
	This concludes the proof since the $\vec w$ live in a linear space over $\mathbb F_5$. 
\end{proof}

This result is very powerful when combined with the following fact.

\begin{lem}
An element $a \in \mathbb Z_5^\times$ is a fifth power precisely if it is so modulo $25$, i.e.~if its reduction $\hat a \in \mathbb Z/25\mathbb Z$ lies in $\{\pm 1, \pm 7\}$.

Hence, the five lifts of any $\bar a \in \left(\mathbb Z/5\mathbb Z\right)^\times$ in $\left(\mathbb Z/25\mathbb Z\right)^\times$ lie in different classes modulo fifth powers.
\end{lem}

To prove Theorem~\ref{thm:invariantmapsfordp525} using brute computational force one would need to list all possible hyperplanes $h$ and points in $\mathcal X(\mathbb Z_5)$ modulo $25$. We will use these last two results to show we can do this computation while only using the points and hyperplane sections modulo $5$; drastically improving on the time needed for the computations.

Let us first show we can ignore the singular point on $\mathcal X_5$, or even the unique line $L \subseteq \mathcal X_5$ containing this point. 

\begin{prop}\label{prop:surjinvmapmodulo5}
If $h$ is not a multiple of $l_1$ modulo $5$ then
	\[
	\inv_5 \mathcal A \colon \mathcal U_h(\mathbb Z_5) \to \frac15\mathbb Z/\mathbb Z
	\]
is surjective.
\end{prop}

\begin{proof}
We will prove a stronger statement. Define $\mathcal V:=\mathcal U_5\backslash L\subseteq \mathbb A^5_{\mathbb F_5}$ on which the hyperplane is given by $h=a_0+a_1u_1+\ldots+a_5u_5$. We will show that there is an $\mathbb F_5$-point $\bar P$ on $\mathcal V$ such that $h \colon T_{\bar P} \mathcal V \to \mathbb F_5$ is surjective; note that one can consider the tangent space as a linear or affine subspace of $\mathbb F_5^5$, since this does not change the size of the image of this function.

One can prove this problem by translating it back to a study of plane curves using the isomorphism $\mathcal V \cong \mathbb A^2_{\mathbb F_5} \backslash \{f=0\}$ for $f=h(1,y,z,y^2,yz,y^3+z^2)$ and conclude that for every $h$ there are at least $10$ such points. For good measure the statement above is checked with \textsc{magma} \cite{Lyczakcode}.

Now let $\bar P$ be such an $\mathbb F_5$-point in $\mathcal V \subseteq \mathcal X$. By the defining property of $\bar P$ we see that $\inv \mathcal A_h$ assumes five values on the points in $\mathcal U_h(\mathbb Z_5)$ reducing to $\bar P$.
\end{proof}

\begin{cor}
If $h \mod 5$ does not cut out the line $L \subseteq \mathcal X_5$ then $\inv_5 \mathcal A$ is surjective.
\end{cor}

We can now efficiently prove Theorem~\ref{thm:invariantmapsfordp525}.

\begin{proof}[Proof of Theorem~\ref{thm:invariantmapsfordp525}]
	By Proposition~\ref{prop:surjinvmapmodulo5} we only need to consider the case that $h$ is not a multiple of $u_0$ modulo $5$ hence we can write $h=\lambda u_0 + 5(c_1u_1+\ldots+c_5u_5) \in \mathbb Z[u_0,u_1,\ldots, u_5]$. Let us write $k=c_1u_1+\ldots+c_5u_5$. Since $l_1 \equiv u_0 \mod 25$ we see that the value of $\frac h{l_1} = \lambda + 5k(\frac{u_1}{u_0},\ldots,\frac{u_5}{u_0}) \mod 25$ at any $P \in \mathcal U(\mathbb Z_5)$ only depends on $\bar P \in \mathcal U(\mathbb F_5)$.

	We are interested in the values $k$ takes on $\mathcal U(\mathbb Z_5)$ modulo $5$. A computer check \cite{Lyczakcode} shows that for the listed cases $k$ assumes the indicated number of values in $\mathbb F_5$. Hence $\frac{h}{l_1}$ assumes the same number of values in $\left(\mathbb Z/25\mathbb Z\right)^\times$ each of which is a different lift of $\lambda \in \mathbb F_5^\times$.  This shows that $\frac{l_1}h$ assumes exactly $1$, $3$ or $5$ values in $\mathbb Z_5^\times$ modulo fifth powers. Hence we see that $\inv_5\mathcal A$ assumes these many values on $\mathcal U_h(\mathbb Z_5)$.

To provide a little more insight we can again use the isomorphism $\mathcal U_5 \cong \mathbb A^2_{\mathbb F_5} \backslash \{f=0\}$ now using $f=k(1,y,z,y^2,yz,y^3+z^2)$. One can check that $f$ is surjective to~$\mathbb F_5^\times$ if it describes a line, a conic with two distinct rational points at infinity, a geometrically integral conic with a single point at infinity, or a cubic curve.  The remaining cases are the constant functions and the quadratics which are independent of~$z$.  This shows that~$k \equiv c_1u_1+c_3u_3 \mod 5$.

The hyperplane section of~$\mathbb P^5_{\mathbb F_5}$ defined by $k\equiv c_1u_1+c_3u_3 \mod 5$ corresponds to the polynomial~$c_1y+c_3y^2$ on~$\mathbb A^2_{\mathbb F_5}$ which is quadratic if~$c_3 \ne 0$ and constant if~$c_1=c_3=0$.  By symmetry we see that a quadratic in one variable over~$\mathbb F_5$ assumes exactly~$3$ values.  And obviously $h \equiv \lambda l_1 \mod 25$ precisely when $c_1$ and $c_3$ are zero in $\mathbb F_5$.
\end{proof}

For completeness we will give an example of a hyperplane for which the associated affine scheme over the integers does not have integral solutions.

\begin{thm}\label{thm:explicitalgebraicobstructiondP525}
	Consider an~$h$ which cuts out a geometrically irreducible hyperplane section such that~$0$ does not lie in the image of~$\inv_5 \mathcal A$.  The reduction of $h$ modulo~$25$ is one of~$176$ out of the~$(5^2)^6-5^6=244125000$ possible hyperplanes over~$\mathbb Z/25\mathbb Z$.  For example, the surface~$\mathcal U_h/\mathbb Z$ for~$h=2u_0-15u_1+10u_3$ admits a Brauer--Manin obstruction of order~$5$ to the existence of integral points.
\end{thm}

\begin{proof}
	Let~$\left(\mathbb Z/25 \mathbb Z \right)^\times$ act by multiplication on the hyperplanes modulo~$25$ for which~$\inv_5 \mathcal A$ is not surjective.  Multiplication by $\lambda$ translates the image of the invariant map by an element of~$\frac15 \mathbb Z/\mathbb Z$ corresponding on the class of $\lambda$ in~$\left(\mathbb Z/25 \mathbb Z \right)^\times$ modulo fifth powers.  So if the size of the image of an invariant map corresponding to a hyperplane has one element, then~$\frac 45$ of the scalar multiples of~$h$ do not have~$0$ in the image.  For invariant maps whose image is of size~$3$ precisely~$\frac25$ of the scalar multiples have this property.
	This means that the number of hyperplanes modulo~$25$ for which~$0$ does not lie in the image of the invariant map is~$\frac 45 \cdot 20+ \frac25 \cdot 20\cdot 4 \cdot 5=176$. 
	
	Now consider the hyperplane~$h=2u_0-15u_1+10u_3$.  The affine surface~$\mathcal U_h$ is locally soluble by Lemma~\ref{lem:localsolubilitydp525}.  The result follows from the previous theorem; take~$\lambda = 2$,~$c_1=-3$ and~$c_3=2$  and note that~$2-15x+10x^2$ only assumes the values~$2,12,22 \mod 25$.  So~$0$ does not lie in the image of the invariant map at~$5$ and the invariant maps at the other primes are all constant zero.
\end{proof}

We can also show that in the two explicit examples with splitting fields with conductor~$11$ and~$25$ the absence of integral points is not explained by the principle of \textit{weak obstructions at infinity} as introduced by Jahnel and Schindler in~\cite{JS}.

\end{document}